\documentclass[11pt]{article}

\usepackage[OT1]{fontenc}
\usepackage[margin=1in]{geometry}
\usepackage{amsmath,amssymb,amsthm,mathtools}
\usepackage{natbib}
\usepackage[x11names]{xcolor}
\usepackage{dsfont}
\usepackage[colorlinks=true,citecolor=blue,linkcolor=blue,urlcolor=blue]{hyperref}

\newcommand{\E}{\mathbb{E}}
\newcommand{\mP}{\mathbb{P}}
\newcommand{\R}{\mathbb{R}}
\newcommand{\eps}{\epsilon}
\newcommand{\one}{\mathbf{1}}
\newcommand{\proj}{P_{\one^\perp}}

\newcommand{\calF}{\mathcal{F}}
\newcommand{\SN}{\mathfrak{S}}
\DeclareMathOperator{\Var}{Var}
\DeclareMathOperator{\Cov}{Cov}
\DeclareMathOperator{\Hyp}{Hypergeometric}

\theoremstyle{plain}
\newtheorem{theorem}{Theorem}
\newtheorem{corollary}{Corollary}
\newtheorem{proposition}{Proposition}
\newtheorem{lemma}{Lemma}
\theoremstyle{definition}

\theoremstyle{remark}
\newtheorem{remark}{Remark}

\definecolor{DarkBlue}{rgb}{0,0,0.55}

\title{A Sharper Hoeffding Bound for Weighted Sums of Exchangeable Random Variables}
\author{
Seongchan Lee\textsuperscript{\dag}
\quad
Ilmun Kim\textsuperscript{\dag}
\\[0.6em]
\small \textsuperscript{\dag}Department of Mathematical Sciences, KAIST, South Korea
}
\date{September 29, 2026}

\begin{document}

\maketitle

\begin{abstract}
We prove a Hoeffding-type moment generating function bound for weighted sums of bounded exchangeable random variables centered by their finite-population average. The bound reduces the excess inflation above one in a recent weighted exchangeable Hoeffding inequality from order $(\log N)/N$ to the rate-optimal order $1/N$, with an explicit constant. The proof reduces the problem to Hamming slices, identifies two-level extremizers for the relevant symmetric variational problem, and applies a hypergeometric martingale bound. We also give a lower bound showing that an excess inflation of order $1/N$ is unavoidable.
\end{abstract}

\section{Introduction}

Concentration inequalities for sums of bounded random variables are fundamental tools in probability and statistics. The cornerstone result in this area is Hoeffding's inequality \citep{hoeffding1963probability}: if $Z_1,\ldots,Z_n$ are independent with $Z_i\in[a_i,b_i]$ almost surely and $\E[Z_i]=0$, then for every $\lambda\in\R$,
\begin{align*}
\E\Biggl[\exp\biggl\{\lambda\sum_{i=1}^n Z_i\biggr\}\Biggr]
\le
\exp\biggl\{\frac{\lambda^2}{8}\sum_{i=1}^n (b_i-a_i)^2\biggr\}.
\end{align*}
This sub-Gaussian moment generating function (MGF) bound implies, for every $t>0$, the tail inequality
\begin{align*}
\mP\biggl\{\sum_{i=1}^n Z_i\ge t\biggr\}
\le
\exp\biggl\{-\frac{2t^2}{\sum_{i=1}^n(b_i-a_i)^2}\biggr\},
\end{align*}
which underlies a vast range of applications in learning theory, survey sampling, and sequential analysis. An extension to bounded differences of martingales was established by \citet{azuma1967weighted}, and sharper tail constants for the i.i.d.\ case were subsequently obtained by \citet{bentkus2004hoeffding}.

In many settings, however, the independence assumption is too strong. A natural and practically important relaxation is exchangeability. A sequence
$(X_1,\ldots,X_N)$ is called exchangeable if its joint distribution is invariant under all permutations. Exchangeable sequences arise canonically in
sampling without replacement from a finite population \citep{serfling1974probability}, in randomization-based inference \citep{fisher1935design,lehmann2005testing}, and in conformal prediction \citep{vovk2005algorithmic,angelopoulos2024theoretical}, where exchangeability
is the key structural assumption enabling finite-sample validity guarantees. Although the de Finetti--Hewitt--Savage theorem \citep{definetti1937prevision,hewitt1955symmetric} characterizes infinite exchangeable sequences as mixtures of i.i.d.\ sequences, finite exchangeable
sequences need not admit such a representation \citep{diaconis1980finite}. Indeed, sampling without replacement gives exchangeable finite sequences with
strictly negative pairwise correlations, so classical independent-sequence bounds do not apply directly.

Concentration inequalities for exchangeable bounded variables have therefore attracted considerable attention. For uniformly weighted vectors, bounds
sharper than the i.i.d.\ Hoeffding bound have been established by \citet{serfling1974probability} and refined by \citet{bardenet2015concentration}.
The key observation is that, without replacement, the sum of the $n$ selected items and the sum of the remaining $N-n$ items are negatively correlated,
yielding a finite-population correction of order $(1-n/N)$. For weight vectors, i.e., for $\sum_{i=1}^n w_i X_i$ with arbitrary $w\in\R^n$, the
situation is more subtle because the partial-sum structure no longer applies. The relevant quantity is the projection of the weight vector onto the orthogonal complement of the all-ones vector. \citet{barber2024hoeffding} recently established a Hoeffding-type MGF bound for weighted sums of exchangeable bounded variables \citep[Theorem~3.1]{barber2024hoeffding}\footnote{The theorem in \citet[Theorem~3.1]{barber2024hoeffding} is stated with $\|\widetilde w\|_2^2$. However, the proof yields the stronger variance proxy $\|\proj\widetilde w\|_2^2$ before the final relaxation $\|\proj\widetilde w\|_2^2\le \|\widetilde w\|_2^2$.} with an inflation factor $1+\eps_N$, where $\eps_N=O((\log N)/N)$. In our centered formulation, the statistic $\sum_{i=1}^n w_i(X_i-\bar X_N)$ is invariant under adding constants to the zero-padded weight vector, so the natural variance proxy is
$\|\proj\widetilde w\|_2^2$ (notation made precise in Section~\ref{sec:main-results}). This result has been subsequently extended to tensor- and matrix-valued data by \citet{cheng2026concentration}.

Our main result identifies the correct finite-population scale for this problem. We prove a Hoeffding-type MGF bound with inflation
\begin{align*}
 \Gamma_N=1+\frac{3}{2N}+O(N^{-2}),   
\end{align*}
improving the excess inflation in the recent bound of \citet{barber2024hoeffding} from order $(\log N)/N$ to order $1/N$. We also prove that no uniform bound can avoid an excess inflation of order $1/N$: the optimal constant satisfies $C_N^\star\ge 1+N^{-1}+c_{N}N^{-2}$ for a bounded sequence $(c_N)_{N\ge 2}$. The improvement over Barber's bound \citep{barber2024hoeffding} comes from a different treatment of the optimization over arbitrary weights. To obtain a symmetric bound for arbitrary weights, \citet{barber2024hoeffding} starts from an order-dependent martingale MGF bound for exchangeable sampling, uses a pseudoinverse construction to express the weighted sum in the required sequential form, and then averages over permutations. The harmonic factor $1+\epsilon_N$ arises from the resulting permutation-averaged pseudoinverse identity. In contrast, we first use Proposition~\ref{prop:two-level} to reduce the symmetric variational problem on each Hamming slice exactly to two-level vectors, and only then apply a martingale bound to the resulting hypergeometric MGF problem.

Beyond the asymptotic improvement, Lemma~\ref{lem:gamma-comparison} shows that $\Gamma_N<1+\eps_N$ for every $N\ge3$, with equality at $N=2$.

A complementary recent approach is developed by~\citet{gottschling2026hoeffding}, who derive Hoeffding-type bounds for exchangeable sums through extremal means of the de Finetti mixing measure. Their focus is on unweighted sums, whereas ours is on finite-population constants for weighted sums centered at $\bar{X}_N$. Relatedly, \citet{lin2026bounded} recover the standard Hoeffding bound for zero-sum linear contrasts under infinite exchangeability, with an application to AI benchmark subsampling. Our result gives the corresponding bound without assuming infinite extendibility, up to a $1+O(N^{-1})$ factor.

A Lean 4 formalization of the main results is available in the \href{https://github.com/statchan1106/exchangeable-hoeffding-lean}{GitHub repository}. The remainder of the paper is organized as follows. Section~\ref{sec:main-results} states the main theorem, corollary, and comparison results. Section~\ref{sec:proof-main} proves Theorem~\ref{thm:main}. Section~\ref{app:auxiliary} collects the auxiliary lemmas and their proofs. Section~\ref{app:proof-lower} gives the proof of Proposition~\ref{prop:lower-bound}.

\section{Main Results}\label{sec:main-results}

We now fix notation. Let $N\ge2$, and write $[N]=\{1,\ldots,N\}$. For $n\le N$ and $w\in\R^n$, define the zero-padded vector $\widetilde w=(w_1,\ldots,w_n,0,\ldots,0)\in\R^N$, and let $\one=(1,\ldots,1)^\top\in\R^N$. Define
\begin{align*}
\proj z=z-\frac{1}{N}(\one^\top z)\one
\end{align*}
to be the orthogonal projection onto the hyperplane $\one^\perp=\{z\in\R^N:\sum_{i=1}^N z_i=0\}$.

\begin{theorem}\label{thm:main}
Let $X_1,\ldots,X_N\in[-1,1]$ be exchangeable and let $\bar X_N=N^{-1}\sum_{i=1}^N X_i$. Then for all $n\le N$, all $w\in\R^n$, and all $\lambda\in\R$,
\begin{align*}
\E\Biggl[\exp\biggl\{\lambda\sum_{i=1}^n w_i(X_i-\bar X_N)\biggr\}\Biggr]
\le
\exp\Bigl\{\frac{\lambda^2}{2}\Gamma_N\|\proj\widetilde w\|_2^2\Bigr\},
\end{align*}
where
\begin{align*}
\Gamma_N
=
\max_{1\le s\le \lfloor N/2\rfloor}
\frac{N(N-s)}{s}
\sum_{\ell=N-s}^{N-1}\frac{1}{\ell^2}
=
\begin{cases}
\displaystyle
N\sum_{\ell=N/2}^{N-1}\ell^{-2}, & N \text{ is even},\\[1.2em]
\displaystyle
\frac{N(N+1)}{N-1}\sum_{\ell=(N+1)/2}^{N-1}\ell^{-2}, & N \text{ is odd}.
\end{cases}
\end{align*}
Moreover, $\Gamma_N=1+3/(2N)+O(N^{-2})$.
\end{theorem}

The proof of Theorem~\ref{thm:main}, including the derivation of the closed form and asymptotic expansion of $\Gamma_N$, is given in Section~\ref{sec:proof-main}.

Optimizing $\lambda$ in Theorem~\ref{thm:main} via Markov's inequality yields the following one-sided tail bound.

\begin{corollary}\label{cor:tail}
Under the assumptions of Theorem~\ref{thm:main}, if $\proj\widetilde w\ne0$, then for any $\delta\in(0,1)$,
\begin{align*}
\mP\biggl\{
\sum_{i=1}^n w_i(X_i-\bar X_N)
\ge
\|\proj\widetilde w\|_2\sqrt{2\Gamma_N\log(1/\delta)}
\biggr\}
\le \delta.
\end{align*}
\end{corollary}

\paragraph{Comparison with Theorem~3.1 of \citet{barber2024hoeffding}.} Their result gives a Hoeffding-type bound with inflation factor $1+\eps_N$,
where
\begin{align*}
\eps_N=\frac{H_N-1}{N-H_N},
\qquad
H_N=\sum_{j=1}^N \frac{1}{j}.
\end{align*}
Since
\begin{align*}
1+\eps_N
=
1+\frac{\log N+\gamma_{\mathrm{EM}}-1}{N}
+
O\Bigl(\frac{(\log N)^2}{N^2}\Bigr),
\end{align*}
where $\gamma_{\mathrm{EM}}\approx0.5772$ is the Euler--Mascheroni constant, their excess inflation satisfies
\begin{align*}
\eps_N=O\Bigl(\frac{\log N}{N}\Bigr),
\end{align*}
whereas Theorem~\ref{thm:main} achieves
\begin{align*}
\Gamma_N-1=\frac{3}{2N}+O(N^{-2}).
\end{align*}
Moreover, Lemma~\ref{lem:gamma-comparison} shows that $\Gamma_N<1+\eps_N$ for every $N\ge3$, with equality at $N=2$.

\paragraph{Rate Optimality.}
Let $C_N^\star$ denote the smallest constant for which $\Gamma_N$ can be replaced in Theorem~\ref{thm:main} uniformly over all exchangeable $X_1,\ldots,X_N\in[-1,1]$, all weights, and all $\lambda\in\R$.

\begin{proposition}\label{prop:lower-bound}
\begin{align*}
C_N^\star
\ge
\begin{cases}
\displaystyle \frac{N}{N-1}, & N \text{ even}, \\[0.8em]
\displaystyle \frac{N+1}{N}, & N \text{ odd}.
\end{cases}
\end{align*}
In particular, $C_N^\star\ge 1+N^{-1}+c_{N}N^{-2}$ for a bounded sequence $(c_N)_{N\ge 2}$, where $c_N=N/(N-1)$ for even $N$ and $c_N=0$ for odd $N$. Thus the first-order coefficient $3/2$ in Theorem~\ref{thm:main} can at best be reduced to $1$.
\end{proposition}

Together, Theorem~\ref{thm:main} and Proposition~\ref{prop:lower-bound} show that the excess inflation $\Gamma_N-1=O(N^{-1})$ is rate-optimal. The bounds also show that any first-order coefficient, if it exists, must lie between $1$ and $3/2$. Whether such a coefficient exists and, if so, its value is left for future work. Proposition~\ref{prop:lower-bound} uses only the second-order behavior of the log-MGF as $\lambda\to0$, yielding a variance-level necessary condition. Therefore, closing the remaining gap may require genuinely nonlocal control of the MGF.

The rate-optimality statement above concerns the projected variance proxy $\|\proj\widetilde w\|_2^2$ uniformly over arbitrary weights. This does not preclude sharper bounds under additional structure. For example, \citet[Section~3.1.1]{barber2024hoeffding} shows that the inflation can be avoided for nonnegative weights, while unweighted sums admit finite-population corrections~\citep{serfling1974probability,bardenet2015concentration}. The proof of Proposition~\ref{prop:lower-bound} is given in Section~\ref{app:proof-lower}.

\begin{remark}\label{remark}
The optimal constant $C_N^\star$ itself admits an exact characterization in terms of hypergeometric MGFs. Let $H_{N,k,m}\sim\Hyp(N,k,m)$ and define
\begin{align*}
\mathcal V_N
:=
\underbrace{
\max_{\substack{1\le k\le N-1\\1\le m\le N-1}}
\frac{8N}{m(N-m)}
\sup_{t\ne0}
\frac{\psi_{N,k,m}(t)}{t^2}
}_{\text{hypergeometric optimization}},
\end{align*}
where $\psi_{N,k,m}(t)$ is the log-MGF of the centered hypergeometric random variable:
\begin{align*}
\psi_{N,k,m}(t)
=
\log\E\Biggl[
\exp\biggl\{
t\Bigl(H_{N,k,m}-\frac{km}{N}\Bigr)
\biggr\}
\Biggr].
\end{align*}
Then $C_{N}^{\star}=\mathcal{V}_N$. This identity follows from the exact reductions used in the proof of Theorem~\ref{thm:main} and is verified immediately after that proof, where the roles of $k$, $m$, and $t$ are made explicit. This characterization provides a concrete formulation for studying the remaining gap in the first-order constant.
\end{remark}

\begin{remark}
Our proof is specific to the scalar Hoeffding setting. It is therefore natural to ask whether the same structural reduction can be extended to other concentration inequalities. We briefly comment on several directions below.
\begin{itemize}
  \item \emph{Tensor and matrix bounds.} For a mode of size $M$, the Hoeffding bound of Cheng and Barber~\citep[Theorem~1]{cheng2026concentration} for mode-exchangeable tensors, as defined in \citep[Definition~2.1]{cheng2026concentration}, contains the factor $1+\epsilon_M$. Because their proof proceeds conditionally, one mode at a time, the scalar improvement developed here may apply whenever the corresponding conditional increment can be expressed as a weighted sum of exchangeable scalar variables. Under general tensor weights, however, different mode slices are coupled with different weight slices, so it is unclear whether such modewise improvements can be combined into a sharper bound for the full tensor. For matrix-valued data, the relevant MGF arguments involve trace exponentials, and the scalar symmetric structure underlying our two-level reduction has no immediate analogue in the noncommutative setting. Thus, neither extension follows directly from our argument, and both appear to require additional structural ideas.
  \item \emph{Bernstein bounds.} Bernstein-type inequalities present a different difficulty. In the Hoeffding setting considered here, the reduction is carried out while preserving the centering constraint and the $\ell_2$ norm, which is sufficient for the resulting MGF bound. Bernstein-type bounds, by contrast, also depend on range and conditional or empirical variance terms. The coordinate perturbations used in our two-level reduction need not preserve these additional quantities, so the same reduction cannot be applied directly. Extending the argument to the Bernstein setting would therefore require a different structural reduction that preserves, or otherwise controls, the relevant range and variance terms together with the $\ell_2$ norm. The feasibility of such a reduction and its potential to yield a sharper finite-population factor remain open questions.
\end{itemize}
\end{remark}

\section{Proof of Theorem~\ref{thm:main}}\label{sec:proof-main}
\begin{proof}
The proof proceeds in three steps. First, we symmetrize over permutations to reduce the problem to the Hamming-slice inequality~\eqref{eq:slice}. Second, Proposition~\ref{prop:two-level} reduces the optimization on each slice to two-level vectors, for which the relevant MGF becomes that of a centered hypergeometric random variable. Finally, we apply the hypergeometric MGF bound of Lemma~\ref{lem:hypergeom}, obtained from a sampling-without-replacement martingale, to complete the proof.

\paragraph{Reduction to Hamming slices.}

Set $a=\proj\widetilde w=\widetilde w-\frac{1}{N}(\one^\top\widetilde w)\one$. Then $\sum_{i=1}^n w_i(X_i-\bar X_N)=a^\top X$ and $\sum_{i=1}^N a_i=0$, so
$a\in\one^\perp$. It therefore suffices to establish
\begin{align*}
\E\bigl[e^{\lambda a^\top X}\bigr]
\le
\exp\Bigl\{\frac{\lambda^2}{2}\Gamma_N\|a\|_2^2\Bigr\} \quad \text{for all $a\in\one^\perp$.}
\end{align*}

Let $\SN_N$ denote the symmetric group on $[N]$, and for fixed $a\in\one^\perp$ define
\begin{align*}
\Phi(x)=\frac{1}{N!}\sum_{\pi\in\SN_N}\exp\bigl\{\lambda a^\top x_\pi\bigr\},
\qquad x\in[-1,1]^N.
\end{align*}
Since $X$ is exchangeable, $X_\pi\stackrel{d}{=}X$ for every $\pi$, so $\E[e^{\lambda a^\top X}]=\E[\Phi(X)]$. Since $\Phi$ is convex on $[-1,1]^N$, its supremum is attained at a vertex $v\in\{\pm1\}^N$, giving $\E[e^{\lambda a^\top X}]\le\max_{v\in\{\pm1\}^N}\Phi(v)$.

Let $P=\{j\in[N]:v_j=+1\}$ with $|P|=k$. Under a uniformly random permutation $\pi\in\SN_N$, the set $S_k=\{i\in[N]:(v_\pi)_i=+1\}$ is uniform over all $k$-element subsets of $[N]$. Since $\sum_{i=1}^N a_i=0$, we have $a^\top v_\pi=2\sum_{i\in S_k}a_i$. Substituting $y_i=2\lambda a_i$, it suffices to prove the slice inequality
\begin{align}
\log\E_{S_k}\Biggl[\exp\biggl\{\sum_{i\in S_k}y_i\biggr\}\Biggr]
\le
\frac{\Gamma_N}{8}\|y\|_2^2,
\qquad
\sum_{i=1}^N y_i=0,
\label{eq:slice}
\end{align}
for each $1\le k\le N-1$ (the boundary cases $k=0,N$ are trivial).

\paragraph{Two-level reduction on the slice.}

Fix $1\le k\le N-1$. Since $S_k$ is uniformly distributed over all $k$-element subsets of $[N]$, we have
\begin{align*}
\E_{S_k}\Biggl[\exp\biggl\{\sum_{i\in S_k}y_i\biggr\}\Biggr]
=
\frac{e_k(e^{y_1},\ldots,e^{y_N})}{\binom{N}{k}},
\end{align*}
where $e_k(z_1,\ldots,z_N)=\sum_{|S|=k}\prod_{i\in S}z_i$ denotes the $k$th elementary symmetric polynomial. Define
\begin{align*}
F_{N,k}(y)=\log\biggl\{\frac{e_k(e^{y_1},\ldots,e^{y_N})}{\binom{N}{k}}\biggr\}.
\end{align*}
By Proposition~\ref{prop:two-level} (proved in Section~\ref{app:auxiliary}), the maximum of $F_{N,k}$ over the sphere section $\{y\in\R^N:\sum_{i=1}^N y_i=0,\ \|y\|_2=\rho\}$ is attained by a vector with at most two distinct coordinate values. It therefore suffices to prove \eqref{eq:slice} for vectors of the form
\begin{align*}
y_i=
\begin{cases}
\alpha, & i\in A,\\
\beta, & i\notin A,
\end{cases}
\qquad
1\le m:=|A|\le N-1,
\qquad
m\alpha+(N-m)\beta=0.
\end{align*}
Set $d=\alpha-\beta$. Then
\begin{align}
\alpha=\frac{N-m}{N}d,
\qquad
\beta=-\frac{m}{N}d,
\qquad
d^2=\rho^2\frac{N}{m(N-m)}.
\label{eq:d-value}
\end{align}
Moreover,
\begin{align*}
\sum_{i\in S_k}y_i
=
d\Bigl(|S_k\cap A|-\frac{km}{N}\Bigr).
\end{align*}
By symmetry of the hypergeometric distribution in $k$ and $m$, the problem therefore reduces to bounding the MGF of
\begin{align*}
H=|S_k\cap A|\sim\Hyp(N,k,m),
\qquad
\E[H]=\frac{km}{N},
\end{align*}
under the convention of Lemma~\ref{lem:hypergeom}.

Applying Lemma~\ref{lem:hypergeom} with $\mu=k/N$ and $t=d$, and writing $s=m\wedge(N-m)$ so that $m(N-m)=s(N-s)$, we obtain
\begin{align*}
\log\E_{S_k}\Biggl[\exp\biggl\{\sum_{i\in S_k}y_i\biggr\}\Biggr]
&\le
\frac{d^2}{8}B_{N,m} =
\frac{\rho^2}{8}
\frac{N(N-s)}{s}
\sum_{\ell=N-s}^{N-1}\frac{1}{\ell^2} \\
&\le
\frac{\Gamma_N}{8}\|y\|_2^2,
\end{align*}
where the last inequality follows from the definition of $\Gamma_N$. This proves \eqref{eq:slice}. Returning to $y=2\lambda a$ completes the MGF bound.

\paragraph{Closed form and asymptotics of $\Gamma_N$.}\label{app:gamma}

We now evaluate the maximum in the definition of $\Gamma_N$ in Theorem~\ref{thm:main}. The case $N=2$ is immediate: the only value is $s=1$, giving $\Gamma_2=2$. It remains to treat $N\ge3$.

Substituting $r=N-s$, the range $1\le s\le\lfloor N/2\rfloor$ becomes $\lceil N/2\rceil\le r\le N-1$. Since $N$ is constant with respect to $r$, the quantity to be maximized is, up to the constant factor $N$,
\begin{align*}
g_r=
\frac{r}{N-r}\sum_{\ell=r}^{N-1}\frac{1}{\ell^2}.
\end{align*}
We claim that $g_r$ is strictly decreasing in $r$. Write $S_r=\sum_{\ell=r}^{N-1}\ell^{-2}$. A direct calculation yields
\begin{align*}
g_{r-1}-g_r
=
\frac{1}{(r-1)(N-r+1)}
-
\frac{N S_r}{(N-r)(N-r+1)}.
\end{align*}
Thus $g_{r-1}>g_r$ holds provided
\begin{align*}
S_r<\frac{N-r}{N(r-1)}.
\end{align*}
Since $x\mapsto x^{-2}$ is strictly convex, the midpoint rule gives a strict inequality:
\begin{align*}
S_r
<
\int_{r-1/2}^{N-1/2}\frac{dx}{x^2}
=
\frac{N-r}{(r-1/2)(N-1/2)}.
\end{align*}
The inequality $(r-1/2)(N-1/2)\ge N(r-1)$ holds because
\begin{align*}
(r-1/2)(N-1/2)-N(r-1)
=
\frac{1}{2}(N-r)+\frac{1}{4}>0.
\end{align*}
Therefore,
\begin{align*}
S_r
<
\frac{N-r}{N(r-1)}.
\end{align*}
Consequently, $g_{r-1}>g_r$, so $g_r$ is strictly decreasing, the maximum is attained at $s=\lfloor N/2\rfloor$, and the closed form stated in Theorem~\ref{thm:main} follows.

\noindent\emph{Even case.} For $N=2M$,
\begin{align*}
\Gamma_N=2M\sum_{\ell=M}^{2M-1}\ell^{-2}.
\end{align*}
The Euler--Maclaurin formula gives
\begin{align*}
\sum_{\ell=M}^{2M-1}\ell^{-2}
=\frac{1}{2M}+\frac{3}{8M^2}+O(M^{-3}),
\end{align*}
so
\begin{align*}
\Gamma_N
=1+\frac{3}{4M}+O(M^{-2})
=1+\frac{3}{2N}+O(N^{-2}).
\end{align*}

\medskip

\noindent\emph{Odd case.} For $N=2M+1$,
\begin{align*}
\Gamma_N
=\frac{(2M+1)(2M+2)}{2M}
\sum_{\ell=M+1}^{2M}\ell^{-2}.
\end{align*}
The Euler--Maclaurin formula gives
\begin{align*}
\sum_{\ell=M+1}^{2M}\ell^{-2}
=\frac{1}{2M}-\frac{3}{8M^2}+O(M^{-3}).
\end{align*}
Using $(2M+1)(2M+2)=4M^2+6M+2$,
\begin{align*}
\Gamma_N
&=\frac{4M^2+6M+2}{2M}
\biggl(\frac{1}{2M}-\frac{3}{8M^2}+O(M^{-3})\biggr) \\
&=1+\frac{3}{4M}+O(M^{-2}).
\end{align*}
Since $1/M=2/N+O(N^{-2})$, we have
\begin{align*}
\frac{3}{4M}=\frac{3}{2N}+O(N^{-2}),
\end{align*}
matching the even case.

This completes the proof of Theorem~\ref{thm:main}.
\end{proof}

\paragraph{Verification of Remark~\ref{remark}.}
We verify that $C_N^\star=\mathcal V_N$. By the Hamming-slice and two-level reductions, any admissible normalized log-MGF is bounded by that of a two-level vector on some Hamming slice. More precisely, let $k$ denote the slice size, let the two weight levels have multiplicities $m$ and $N-m$, and let $t$ denote their signed difference. For a uniformly random $k$-subset, the corresponding sum depends only on the number of selected coordinates from the level of multiplicity $m$, which has distribution $\Hyp(N,k,m)$. The calculation above then gives
\begin{align*}
F_{N,k}(y)=\psi_{N,k,m}(t),
\qquad
\|y\|_2^2=\frac{m(N-m)}{N}t^2,
\end{align*}
and hence
\begin{align*}
\frac{8F_{N,k}(y)}{\|y\|_2^2}
=
\frac{8N}{m(N-m)}
\frac{\psi_{N,k,m}(t)}{t^2}
\le \mathcal V_N.
\end{align*}
Thus $C_N^\star\le\mathcal V_N$.

Conversely, fix $1\le k,m\le N-1$ and $t\ne0$, let $X$ be uniform on the $k$th Hamming slice, and take $a\in\one^\perp$ equal to $(N-m)t/(2N)$ on a set $A$ of size $m$ and $-mt/(2N)$ otherwise. Then
\begin{align*}
a^\top X
=
t\Bigl(|S_k\cap A|-\frac{km}{N}\Bigr),
\qquad
\|a\|_2^2=\frac{m(N-m)}{4N}t^2,
\end{align*}
where $|S_k\cap A|\sim\Hyp(N,k,m)$. The defining bound for $C_N^\star$, with $\lambda=1$, gives
\begin{align*}
\frac{8N}{m(N-m)}
\frac{\psi_{N,k,m}(t)}{t^2}
\le C_N^\star.
\end{align*}
Taking the supremum over $t\ne0$ and the maximum over $k,m$ yields $\mathcal V_N\le C_N^\star$, completing the verification.

\section{Auxiliary Results and Proofs}\label{app:auxiliary}

The proof of Theorem~\ref{thm:main} uses several auxiliary results whose proofs were deferred. We collect these auxiliary results and their proofs here.

We begin with the classical Hoeffding lemma~\citep{hoeffding1963probability}, which is used in the proof of Lemma~\ref{lem:hypergeom}.
\begin{lemma}\label{lem:hoeffding}
Let $Z$ be a real random variable with $\E[Z]=0$ and $Z\in[a,b]$ a.s. Then for every $\theta\in\R$,
\begin{align*}
\E\bigl[e^{\theta Z}\big]
\le
\exp\Bigl\{\frac{\theta^2(b-a)^2}{8}\Bigr\}.
\end{align*}
\end{lemma}

The following lemma is a technical sign result used in the proof of Lemma~\ref{lem:three-point}. It shows that a certain weighted sum involving Lagrange basis polynomials is strictly positive, a fact invoked via the second-order optimality condition.

\begin{lemma}\label{lem:algebraic}
Let $I\subset\R$ be an interval and $h\in C^5(I)$. Let $x_1,x_2,x_3\in I$ be distinct and set $p(t)=\prod_{i=1}^3(t-x_i)$. If $h(x_i)=0$ for $i=1,2,3$ and $h^{(5)}(t)>0$ on $I$, then $\sum_{i=1}^3 h'(x_i)/p'(x_i)^2>0$.
\end{lemma}

\begin{proof}
Define the Lagrange basis polynomials $\ell_i(t)=p(t)/[(t-x_i)p'(x_i)]$, so that $\ell_i(x_j)=\delta_{ij}$, and set $H(t)=\sum_{i=1}^3 h'(x_i)(t-x_i)\ell_i(t)^2$. By construction, $H(x_i)=h(x_i)=0$ and $H'(x_i)=h'(x_i)$ for each $i$, so $h-H$ has three distinct zeros, each of multiplicity at least two. Repeated application of Rolle's theorem therefore yields some $\xi\in(\min_i x_i,\max_i x_i)$ such that $(h-H)^{(5)}(\xi)=0$.

Since $(t-x_i)\ell_i(t)^2$ is a degree-$5$ polynomial with leading coefficient $1/p'(x_i)^2$, the leading coefficient of $H$ is $\sum_{i=1}^3 h'(x_i)/p'(x_i)^2$, and therefore $H^{(5)}(t)=5!\sum_{i=1}^3 h'(x_i)/p'(x_i)^2$ for all $t$. Hence
\begin{align*}
\sum_{i=1}^3\frac{h'(x_i)}{p'(x_i)^2}
=\frac{h^{(5)}(\xi)}{5!}>0,
\end{align*}
which completes the proof.
\end{proof}

The following lemma treats the three-coordinate sphere section that arises in the proof of Proposition~\ref{prop:two-level}: it shows that the relevant symmetric exponential objective is maximized at a point where at least two coordinates coincide.

\begin{lemma}\label{lem:three-point}
Let $A,B\ge0$ with $A+B>0$. Fix $s,q\in\R$ with $q\ge s^2/3$. On
\begin{align*}
C_{s,q}=\{(u,v,z):u+v+z=s,\ u^2+v^2+z^2=q\},
\end{align*}
the function $G(u,v,z)=A(e^u+e^v+e^z)+B(e^{-u}+e^{-v}+e^{-z})$ is maximized at a point where at least two of $u,v,z$ are equal. If $q>s^2/3$, no point with three distinct coordinates can be a maximizer.
\end{lemma}

\begin{proof}
If $q=s^2/3$, then Cauchy--Schwarz forces $u=v=z=s/3$, and the conclusion is immediate. Henceforth, assume $q>s^2/3$.

Set $\phi(t)=Ae^t+Be^{-t}$, so that $G=\phi(u)+\phi(v)+\phi(z)$. Since $C_{s,q}$ is compact, $G$ attains its maximum. Suppose, for a contradiction, that some maximizer $x=(x_1,x_2,x_3)$ has pairwise distinct coordinates.

At any point with distinct coordinates, the gradients of the constraint functions $g_1(x)=\sum_{i=1}^3 x_i$ and $g_2(x)=\sum_{i=1}^3 x_i^2$ are linearly independent, so the Lagrange multiplier theorem furnishes $\alpha,\beta\in\R$ such that $\phi'(x_i)=\alpha+\beta x_i$ for $i=1,2,3$. Setting $h(t)=\phi'(t)-\alpha-\beta t$, one has $h(x_i)=0$ for each $i$ and $h^{(5)}(t)=Ae^t+Be^{-t}>0$ for all $t\in\R$.

Define $p(t)=\prod_{i=1}^3(t-x_i)$ and $\eta_i=1/p'(x_i)$. A direct common-denominator calculation gives $\sum_{i=1}^3 1/p'(x_i)=0$ and $\sum_{i=1}^3 x_i/p'(x_i)=0$, so $\eta=(\eta_1,\eta_2,\eta_3)$ is tangent to $C_{s,q}$ at $x$. The second-order necessary condition for a constrained local maximum reads $\sum_{i=1}^3\{\phi''(x_i)-\beta\}\eta_i^2\le0$. Since $\phi''(x_i)-\beta=h'(x_i)$, the left-hand side equals $\sum_{i=1}^3 h'(x_i)/p'(x_i)^2$, which is strictly positive by Lemma~\ref{lem:algebraic}, a contradiction.
\end{proof}

The following proposition is the central structural result of the proof: it shows that on any Hamming slice, the worst-case input has at most two distinct coordinate values, reducing a high-dimensional continuous optimization to a one-dimensional hypergeometric moment bound.

\begin{proposition}\label{prop:two-level}
Fix $1\le k\le N-1$ and $\rho\ge0$, and define
\begin{align*}
F_{N,k}(y)=\log\biggl\{\frac{e_k(e^{y_1},\ldots,e^{y_N})}{\binom{N}{k}}\biggr\}.
\end{align*}
The maximum of $F_{N,k}$ over $\{y\in\R^N:\sum_{i=1}^N y_i=0,\ \|y\|_2=\rho\}$ is attained by a vector with at most two distinct coordinate values.
\end{proposition}

\begin{proof}
The constraint set is compact, so a maximizer exists; $\rho=0$ is immediate, so assume $\rho>0$. Let $y$ be a maximizer, and suppose for a contradiction that $y$ takes three or more distinct values. Since the logarithm is strictly increasing and $\binom{N}{k}$ is constant, maximizing $F_{N,k}$ is equivalent to maximizing $e_k(e^{y_1},\ldots,e^{y_N})$. Choose three coordinates at distinct values, label them $u,v,z$, and write $R$ for the remaining $N-3$ entries. With the convention $e_j(e^R)=0$ outside $0\le j\le N-3$, set $C_j=e_{k-j}(e^R)\ge0$ for $j=0,1,2,3$. Then
\begin{align*}
e_k(e^{y_1},\ldots,e^{y_N})
=C_0+C_1(e^u+e^v+e^z)+C_2(e^{u+v}+e^{u+z}+e^{v+z})+C_3e^{u+v+z}.
\end{align*}
Since $y$ is a global maximizer, the triple $(u,v,z)$ must also maximize the three-variable objective over $\{u+v+z=s,\ u^2+v^2+z^2=q\}$ with $R$ fixed. Since $u,v,z$ are not all equal, Cauchy--Schwarz gives $q>s^2/3$.

In this sphere section, $C_0$ and $C_3e^s$ are constant. Using $e^{u+v}+e^{u+z}+e^{v+z}=e^s(e^{-u}+e^{-v}+e^{-z})$, the non-constant part reduces to $C_1(e^u+e^v+e^z)+C_2e^s(e^{-u}+e^{-v}+e^{-z})$.

We verify $C_1+C_2>0$: if $k=1$, then $C_1=e_0(e^R)=1>0$; if $k=N-1$, then $C_2=e_{N-3}(e^R)>0$; and if $2\le k\le N-2$, both $e_{k-1}(e^R)$ and $e_{k-2}(e^R)$ are nonzero elementary symmetric polynomials of strictly positive numbers (since $0\le k-2\le N-4$ and $1\le k-1\le N-3$ are both admissible degrees for a vector of length $N-3$). Thus $C_1+C_2>0$, so Lemma~\ref{lem:three-point} applies with $A=C_1$ and $B=C_2e^s$, implying that no maximizer can have three distinct coordinates, a contradiction.
\end{proof}

The two-level reduction leaves a one-dimensional hypergeometric problem. The normalized sampling-without-replacement martingale used below is standard; see, for example, \citep{bardenet2015concentration}. Lemma~\ref{lem:hypergeom} records the resulting MGF bound in the form needed here.

\begin{lemma}\label{lem:hypergeom}
Let $H\sim\mathrm{Hypergeometric}(N,K,m)$ with $0\le K\le N$ and $1\le m\le N-1$, and let $\mu=K/N$. Set $s=m\wedge(N-m)$. Then for every $t\in\R$,
\begin{align*}
\log\E\Bigl[\exp\bigl\{t(H-m\mu)\bigr\}\Bigr]
\le
\frac{t^2}{8}B_{N,m},
\qquad
B_{N,m}=(N-s)^2\sum_{\ell=N-s}^{N-1}\frac{1}{\ell^2}.
\end{align*}
\end{lemma}

\begin{proof}
Let $(x_1,\ldots,x_N)\in\{0,1\}^N$ satisfy $N^{-1}\sum_{i=1}^N x_i=\mu$. Let $X_1,\ldots,X_m$ be drawn sequentially without replacement, set $D_i=X_i-\mu$ and $S_j=\sum_{i=1}^j D_i$, and let $\calF_j=\sigma(X_1,\ldots,X_j)$ with $\calF_0$ trivial. Then $H-m\mu=S_m$.

\emph{Case $m\le N/2$.} Define $M_j=S_j/(N-j)$. After $j-1$ draws, the remaining centered population sum is $-S_{j-1}$, so $\E[D_j\mid\calF_{j-1}]=-S_{j-1}/(N-j+1)=-M_{j-1}$, and hence
\begin{align*}
M_j-M_{j-1}=\frac{D_j+M_{j-1}}{N-j}.
\end{align*}
Conditional on $\calF_{j-1}$, the random variable $D_j+M_{j-1}$ has mean zero and is supported on an interval of length at most $1$. Hence Lemma~\ref{lem:hoeffding} yields
\begin{align*}
\E\Bigl[e^{\theta(M_j-M_{j-1})}\,\big|\,\calF_{j-1}\Bigr]
\le
\exp\Bigl\{\frac{\theta^2}{8(N-j)^2}\Bigr\}.
\end{align*}
Iterating the conditional bound via the law of total expectation gives
\begin{align*}
\E\bigl[e^{\theta M_m}\bigr]
\le
\exp\biggl\{\frac{\theta^2}{8}\sum_{j=1}^m\frac{1}{(N-j)^2}\biggr\}.
\end{align*}
Since $S_m=(N-m)M_m$, taking $\theta=t(N-m)$ and reindexing $\ell=N-j$ yields
\begin{align*}
\log\E\bigl[e^{t S_{m}}\bigr]
\le
\frac{t^2}{8}(N-m)^2\sum_{\ell=N-m}^{N-1}\frac{1}{\ell^2}.
\end{align*}

\emph{Case $m>N/2$.} The complement of a uniformly sampled $m$-subset is a uniformly sampled $(N-m)$-subset. Since the full centered population sum is zero, the centered sum over the original sample equals the negative of the centered sum over its complement. Since $N-m<N/2$, applying the preceding case to this complement gives
\begin{align*}
\log\E\bigl[e^{t(H-m\mu)}\bigr]
\le
\frac{t^2}{8}m^2\sum_{\ell=m}^{N-1}\frac{1}{\ell^2}.
\end{align*}
Combining the two cases via $s=m\wedge(N-m)$ yields the stated bound.
\end{proof}

The following lemma shows that Theorem~\ref{thm:main} strictly improves upon the bound of \citet{barber2024hoeffding} for every $N\ge3$, confirming that the improvement is not merely asymptotic.

\begin{lemma}\label{lem:gamma-comparison}
For every $N\ge3$,
\begin{align*}
\Gamma_N < 1+\epsilon_N,
\end{align*}
where $\epsilon_N=(H_N-1)/(N-H_N)$ and $H_N=\sum_{j=1}^N j^{-1}$. For $N=2$, equality holds.
\end{lemma}

\begin{proof}
First observe that $1+\epsilon_N=\frac{N-1}{N-H_N}$. We use the algebraic inequality
\begin{align*}
\frac{1}{\ell^2} < \frac{1}{\ell^2 - 1/4} = \frac{1}{\ell - 1/2} - \frac{1}{\ell + 1/2},
\end{align*}
which holds for all $\ell \ge 1$. This yields a telescoping sum, avoiding the need for integral bounds. We prove the claim according to the parity of $N$.
\begin{enumerate}
    \item[(1)]\textbf{Even case.} By the closed form of $\Gamma_N$,
    \begin{align*}
    \Gamma_N = N\sum_{\ell=N/2}^{N-1}\frac{1}{\ell^2}
    &< N\sum_{\ell=N/2}^{N-1}
      \biggl(\frac{1}{\ell-1/2}-\frac{1}{\ell+1/2}\biggr) \\
    &= N\biggl(\frac{2}{N-1}-\frac{1}{N-1/2}\biggr)
     = \frac{2N^2}{(N-1)(2N-1)}.
    \end{align*}
    It therefore suffices to show $\frac{2N^2}{(N-1)(2N-1)} < \frac{N-1}{N-H_N}$, which is equivalent to $H_N > \frac{5}{2} - \frac{2}{N} + \frac{1}{2N^2}$.
    \item[(2)] \textbf{Odd case.} By the closed form of $\Gamma_N$,
    \begin{align*}
    \Gamma_N
    &= \frac{N(N+1)}{N-1}
      \sum_{\ell=(N+1)/2}^{N-1}\frac{1}{\ell^2} \\[-0.2em]
    &< \frac{N(N+1)}{N-1}
      \sum_{\ell=(N+1)/2}^{N-1}
      \biggl(\frac{1}{\ell-1/2}-\frac{1}{\ell+1/2}\biggr) \\[-0.2em]
    &= \frac{N(N+1)}{N-1}
      \biggl(\frac{2}{N}-\frac{1}{N-1/2}\biggr)
     = \frac{2(N+1)}{2N-1}.
    \end{align*}
    It therefore suffices to show $\frac{2(N+1)}{2N-1} < \frac{N-1}{N-H_N}$, which is equivalent to $H_N > \frac{5}{2} - \frac{3}{N+1}$.
\end{enumerate}
Notice that for $N \ge 2$, the terms $(-\tfrac{2}{N} + \tfrac{1}{2N^2})$ and $(- \tfrac{3}{N+1})$ are strictly negative. Therefore, the single condition $H_N \ge 5/2$ is sufficient to satisfy the inequalities for both the even and odd cases.

For every $N \ge 7$, we have $H_N \ge H_7 = 363/140 > 5/2$, establishing the inequality for all $N \ge 7$. The remaining cases for $N \in \{3, 4, 5, 6\}$ are checked directly
\begin{align*}
\Gamma_3&=\tfrac{3}{2}<\tfrac{12}{7}=1+\epsilon_3,\qquad
&\Gamma_4&=\tfrac{13}{9}<\tfrac{36}{23}=1+\epsilon_4,\\
\Gamma_5&=\tfrac{125}{96}<\tfrac{240}{163}=1+\epsilon_5,\qquad
&\Gamma_6&=\tfrac{769}{600}<\tfrac{100}{71}=1+\epsilon_6.
\end{align*}
Finally, for $N=2$, the closed form yields $\Gamma_2=2=1+\epsilon_2$, satisfying the equality. This completes the proof.
\end{proof}

\subsection{Proof of Proposition~\ref{prop:lower-bound}}\label{app:proof-lower}

\begin{proof}[Proof of Proposition~\ref{prop:lower-bound}]
Fix a nonzero vector $a\in\one^\perp$. Fix $1\le k\le N-1$, and let $X$ be uniformly distributed on the Hamming slice $\{x\in\{\pm1\}^N:\#\{i:x_i=1\}=k\}$, which is exchangeable by symmetry. This choice is admissible in Theorem~\ref{thm:main} by taking $n=N$ and $w=a$, since $a\in\one^\perp$ implies $\proj w=a$. Writing $S_k=\{i:X_i=1\}$ gives
\begin{align*}
\sum_{i=1}^N a_iX_i
=
2\sum_{i\in S_k}a_i
=:
Y.
\end{align*}
Since $a\in\one^\perp$, we have $\E Y=0$. Hence, as $\lambda\to0$, $\log\E e^{\lambda Y}=\frac{\lambda^2}{2}\Var(Y)+o(\lambda^2)$. Therefore, if the bound $\log\E e^{\lambda Y}\le\frac{\lambda^2}{2}C_N^\star\|a\|_2^2$ holds for all $\lambda\in\R$, then taking $\lambda\to0$ gives $\Var(Y)\le C_N^\star\|a\|_2^2$.

Denoting by $I_i=\mathds{1}(i\in S_k)$ the membership indicators, the hypergeometric variance formulas give
\begin{align*}
\Var(I_i)=\frac{k(N-k)}{N^2},
\qquad
\Cov(I_i,I_j)=-\frac{k(N-k)}{N^2(N-1)},
\qquad i\ne j.
\end{align*}
Since $\sum_{i=1}^N a_i=0$ implies $\sum_{i\ne j}a_i a_j=-\|a\|_2^2$,
\begin{align*}
\Var\biggl(\sum_{i\in S_k}a_i\biggr)
=\frac{k(N-k)}{N(N-1)}\|a\|_2^2,
\end{align*}
and hence $\Var(Y)=4k(N-k)\|a\|_2^2/[N(N-1)]$. Since $k\in\{1,\ldots,N-1\}$ was arbitrary,
\begin{align*}
C_N^\star\ge\max_{1\le k\le N-1}\frac{4k(N-k)}{N(N-1)}.
\end{align*}
The maximum is attained at $k=\lfloor N/2\rfloor$, and a direct computation yields the stated bound.
\end{proof}

\paragraph{Acknowledgements.} The authors used AI-assisted tools for brainstorming proof strategies, checking algebraic manipulations, and improving exposition. The authors reviewed and verified all mathematical claims and final wording and take full responsibility for the content.

\bibliographystyle{apalike}
\bibliography{reference}

\end{document}